\documentclass[11pt,reqno]{amsart}
\usepackage{amssymb,latexsym}
\usepackage{graphicx}
\usepackage{mathtools}
\usepackage{color}
\usepackage[hyphens]{url}
\usepackage[T1]{fontenc}
\usepackage{hyperref}
\usepackage{amsmath}
\usepackage{mathdots}
\usepackage{breqn}
\usepackage[toc,page]{appendix}
\usepackage{url}    
\usepackage{breqn}
\usepackage{hyperref}
\usepackage{breakurl}
\newcommand{\bburl}[1]{\textcolor{blue}{\url{#1}}}

\calclayout

\makeatletter
\newcommand{\monthyear}[1]{%
  \def\@monthyear{\uppercase{#1}}}
\newcommand{\volnumber}[1]{%
  \def\@volnumber{\uppercase{#1}}}
\AtBeginDocument{%
\def\ps@plain{\ps@empty
  \def\@oddfoot{\@monthyear \hfil \thepage}%
  \def\@evenfoot{\thepage \hfil \@volnumber}}
\def\ps@firstpage{\ps@plain}
\def\ps@headings{\ps@empty
  \def\@evenhead{%
    \setTrue{runhead}%
    \def\thanks{\protect\thanks@warning}%
    \uppercase{\ }\hfil}%
  \def\@oddhead{%
    \setTrue{runhead}%
    \def\thanks{\protect\thanks@warning}%
    \hfill\uppercase{Higher Order Fibonacci Sequences from Generalized Schreier sets}}%
  \let\@mkboth\markboth
  \def\@evenfoot{%
    \thepage \hfil \@volnumber}%
  \def\@oddfoot{%
    \@monthyear \hfil \thepage}%
  }%
\footskip=25pt
\pagestyle{headings}%
}
\makeatother

\theoremstyle{plain}
\numberwithin{equation}{section}
\newtheorem{thm}{Theorem}[section]

\newcommand{\ignore}[1]{}

\newcommand\be{\begin{eqnarray}}
\newcommand\ee{\end{eqnarray}}
\newcommand\bea{\begin{eqnarray}}
\newcommand\eea{\end{eqnarray}}
\newcommand\ben{\begin{enumerate}}
\newcommand\een{\end{enumerate}}

\newtheorem{cor}[thm]{Corollary}

\newtheorem{rek}[thm]{Remark}

\newcommand{\N}{\mathbb{N}}

\begin{document}

\monthyear{September 2019}
\volnumber{Volume, Number}
\setcounter{page}{1}
\title{Higher Order Fibonacci Sequences from Generalized Schreier sets}

\author{H\`ung Vi\d{\^e}t Chu, Steven J. Miller, and Zimu Xiang}

\address{Department of Mathematics, University of Illinois at Urbana-Champaign, Urbana, IL 61820} \email{chuh19@mail.wlu.edu}

\address{Department of Mathematics and Statistics, Williams College, Williamstown, MA 01267} \email{sjm1@williams.edu}

\address{Department of Mathematical Sciences, Carnegie Mellon University, Pittsburgh, PA 15213} \email{zimux@andrew.cmu.edu}

\date{\today}

\begin{abstract}
A Schreier set $S$ is a subset of the natural numbers with $\min S\ge |S|$. It has been known that the sequence $(a_{1,n})$, where
$$a_{1,n}\ :=\ |\{S\subseteq \mathbb{N}\,:\,\max S = n\mbox{ and } \min S \ge |S|\}|,$$ is the Fibonacci sequence. Generalizing this result, we prove that for all $p\in \mathbb{N}$, the sequence $(a_{p,n})$, where $$a_{p, n} \ :=\ |\{S\subseteq \mathbb{N}\,:\,\max S = n\mbox{ and } \min S\ge p|S|\}|,$$ has a linear recurrence relation of higher order. We investigate further by requiring that ${\rm min}_2 S\ge q |S|$, where $\min_2 S$ is the second smallest element of $S$. We prove a linear recurrence relation for the sequence $(a_{p, q, n})$, where $$a_{p, q, n} \ :=\  |\{S\subseteq \mathbb{N}\,:\,\max S = n, \min S \ge p|S|\mbox{ and } {\rm min}_2 S\ge q|S|\}|,$$
and discuss a curious relationship between $(a_{q, n})$ and $(a_{p, q, n})$.
\end{abstract}

\thanks{This work was partially supported by NSF grants DMS1561945 and DMS1659037, Carnegie Mellon University, Washington and Lee University, and Williams College. The authors are thankful for the anonymous referee's insightful comments that improved the paper's content and exposition.}

\maketitle
\section{Introduction}
A Schreier set $S$ is a subset of the natural numbers with $\min S\ge |S|$, and the Schreier family containing all Schreier sets is denoted by $\mathcal{S}_1$. Schreier defined them to solve a problem in Banach space theory in 1930 \cite{Sch}. These sets were also independently discovered in combinatorics and are connected to Ramsey-type theorems for subsets of $\mathbb{N}$. Bird \cite{UA} proved that the Fibonacci sequence appears if we count Schreier sets under certain conditions.

Define \begin{align*} M_{1,n} \ :=\  \{S\in\mathcal{S}_1\,:\,\max S = n\}.\end{align*} Then $|M_{1,1}|= 1$, $|M_{1,2}| = 1$ and $|M_{1,n+2}| = |M_{1,n+1}| + |M_{1,n}|$ for all $n\ge 1$ \cite{UA}. The proof uses two one-to-one mappings to argue about cardinalities of sets. We generalize this result by defining, for $p \in \mathbb{N}$,
\begin{align*} \mathcal{S}_p \ :=\ \{S\subseteq \mathbb{N}:\min S\ge p|S|\}, \ \ \  {\rm and} \ \ \ M_{p,n} \ :=\ \{S\in\mathcal{S}_p\,:\,\max S = n\}, \end{align*} and prove the following\footnote{Our definition of $\mathcal{S}_p$ is not the same as what used in Banach space theory to indicate the Schreier sets of order $p$ \cite{AA}.}.

\begin{thm}\label{mainTheo}
Given $p\in \mathbb{N}$, consider the sequence $(|M_{p,n}|)_{n=1}^\infty$. We have
\begin{enumerate}
\item $|M_{p,n+p}| = \sum_{k=1}^{n+p-1}\sum_{j=0}^{k/p-2}\binom{n+p-k-1}{j}+1$, and

\item for $n\ge 1$, $|M_{p,n+p+1}| = |M_{p,n+p}|+|M_{p,n}|$.
\end{enumerate}
We call $(|M_{p,n}|)_{n=1}^\infty$ the generalized Schreier-Fibonacci sequence of order $p$.
\end{thm}

Another natural extension is to put an additional restriction on our set $S$; in particular, we require that $\min_2 S \ge q|S|$, where $\min_2 S$ is the second smallest element in $S$. We define \begin{align*} \mathcal{S}_{p,q} \ :=\ \{S\subseteq \N\,:\,\min S\ge p|S|\mbox{ and } {\rm min}_2 S\ge q|S|\}.\end{align*}

For a given $n$, we consider the family of sets $M_{p,q, n}= \{S\in \mathcal{S}_{p,q}\,:\,\max S = n\}$. When a set has exactly one element, we take the element to be both the smallest and the second smallest. The following theorem gives an explicit formula to calculate $|M_{p,q, n}|$.

\begin{thm}\label{further}
Given $p < q\in\N$, for the sequence $(|M_{p,q,n}|)_{n\ =\ 1}^\infty$, we have $|M_{p,q, n}| = 0$ if
$n\le q-1$, $|M_{p,q, n}| = 1$ if $q\le n\le 2q-1$ and
\begin{equation*}
|M_{p,q, n}|\ =\ 1+(n-2p)+\sum_{k\ =\ 3}^{\frac{n+2}{q+1}}\sum_{i\ =\ q k}^{n+2-k}(i-p k)\binom{n-i-1}{k-3} \mbox{{\rm\ if\ } }n\ge 2q.
\end{equation*}
\end{thm}

\begin{thm}
\label{keylem} Fix $p < q \in \mathbb{N}$. Consider $(M_{q,n})_{n=1}^\infty$ and $(M_{p, q, n})_{n=1}^\infty$. For each $n\in\mathbb{N}$, define $a_n := |M_{p,q,n+q}|$.
We have
\begin{align*}
a_{n+q+1}\ =\ a_{n+q} + a_n+(q-p)|M_{q,n}|.
\end{align*}
\end{thm}
Note that when $p=q$, we have Theorem \ref{mainTheo}.
We have the following corollary that shows a recurrence relation for the sequence $(|M_{p,q,n}|)_{n=1}^\infty$.
\begin{cor}\label{linear2}
Fix $p<q$ in $\mathbb{N}$. For $n\in\mathbb{N}$, define $a_n:=|M_{p,q,n+q}|$. We have
\begin{align*}
a_{n+2q+2}\ =\ 2a_{n+2q+1}-a_{n+2q}+2a_{n+q+1}-2a_{n+q}-a_n.
\end{align*}
\end{cor}

\begin{proof}
Fix $n\in\mathbb{N}$. By Theorem \ref{keylem}, we have
\begin{align}
      \label{first}a_{n+q+1}-a_{n+q}&\ =\ a_n+(q-p)|M_{q,n}| \\
      \label{second}a_{n+2q+1}-a_{n+2q}&\ =\ a_{n+q}+(q-p)|M_{q,n+q}| \\
      \label{last}a_{n+2q+2}-a_{n+2q+1}&\ =\ a_{n+q+1}+(q-p)|M_{q,n+q+1}|.
\end{align}
By Theorem \ref{mainTheo}, we know that $|M_{q,n+q+1}|=|M_{q,n+q}|+|M_{q,n}|$. Subtract Equation (\ref{first}) and Equation (\ref{second}) from Equation (\ref{last}) to finish the proof.
\end{proof}

\begin{rek} \normalfont For fixed $p, q$, Theorem \ref{linear2} gives a recurrence relation of depth $2q+2$; interestingly, the depth is independent of $p$.
\end{rek}
\section{Proof of Theorem \ref{mainTheo}}
Given a set $S$ and a number $a$, define
$$a+S \ := \ \{a+s\,:\, s\in S\}.$$
In our proof, we partition $M_{p, n+p+1}$ into two disjoint sets $A$ and $B$ then use bijective maps to show that $|A| = |M_{p, n+p}|$ and $|B| = |M_{p, n}|$. This is the same technique used in \cite{UA}.

\begin{proof}[Proof of Theorem \ref{mainTheo}] \ \\

\noindent (1) To find an explicit formula for $|M_{p, n+p}|$, we use the following simple counting argument. Let $k$ be the minimum element of our set $S\in M_{p, n+p}$. If $k = n+p$, then $S$ $=$ $\{n+p\}$. If $k< n+p$, then we can choose it to be any number between $1$ and $n+p-1$. For each of these choices, we have fixed the maximum and the minimum of our set and so, we can choose $j$ elements between $k+1$ and $n+p-1$, where $j\le k/p-2$. Therefore
$$|M_{p,n+p}| \ =\ \sum_{k=1}^{n+p-1}\sum_{j=0}^{k/p-2}\binom{n+p-k-1}{j}+1,$$ which is the desired formula. \\ \

\noindent (2) The set $M_{p, n+p +1}$ is the union of
\begin{itemize}
    \item [(a)] $A = \{S\in M_{p, n+p+1}: n+p\notin S\}$,
    \item [(b)] $B = \{S\in M_{p, n+p+1}: n+p\in S\}$.
\end{itemize}
We compute $|A|$ by considering the map $R_1: M_{p, n+p}\rightarrow A$ with $R_1(S) = (S\backslash\{n+p\})\cup \{n+p+1\}$. The map is well-defined because it preserves the cardinality of the set and does not decrease the minimum element of a set. Injectivity of $R_1$ is clear. The map is also onto because given $U\in A$, $R_1(U\backslash\{n+p+1\}\cup \{n+p\}) = U$. So, $|A| = |M_{p, n+p}|$.

Next, we determine $|B|$ by considering the map $R_2: M_{p, n}\rightarrow B$ with $R_2(S) = (S+p)\cup \{n+p+1\}$. Since $\min S\ge p |S|$, $\min (S+p)\ge p(|S|+1)$. This shows that $R_2$ is well-defined. Injectivity is clear. The map is also onto because given $U\in B$, $R_2((U\backslash \{n+p+1\})-p) = U$. So, $|B| = |M_{p, n}|$. We conclude that $$|M_{p,n+p+1}|\ \ =\ \ |M_{p,n+p}|\ +\ |M_{p,n}|.$$
\end{proof}

\section{Proof of Theorem \ref{further} and Theorem \ref{keylem}}

Our proof of Theorem \ref{further} employs straightforward counting arguments. For Theorem \ref{keylem}, we partition $M_{p,q,n+2q+1}$ into three subsets and use bijective maps to argue that the cardinalities of these three subsets are equal to $a_{n+q}, a_n$ and $(q-p)|M_{q,n}|$, respectively.

\begin{proof}[Proof of Theorem \ref{further}]
Fix $p < q \in \mathbb{N}$. We prove the theorem by considering different ranges for $n$. For $n \le q - 1$, if $|S| > 0$ we have
the contradiction $$q \ \le\ q|S| \ \le\ {\rm min}_2 S \ \le\ n \ \le\ q - 1.$$ For $q \le n \le 2q - 1$, we have
$|S| = 1$ since otherwise we have the contradiction $$2q \ \le\ q|S| \ \le\ {\rm min}_2 S \ \le\ n \ \le\ 2q - 1.$$ If $n\ge 2q$, we prove that \begin{align*}\label{baseeq}|M_{p,q, n}|\ =\ 1+(n-2p)+\sum_{k\ =\ 3}^{\frac{n+2}{q+1}}\sum_{i\ =\ q k}^{n-k+2}(i-p k)\binom{n-i-1}{k-3}.\end{align*}

\begin{itemize}
\item The $1$ on the right hand side comes from the set $\{n\}$. \\ \

\item For a two-element set $S$, the maximum element $n$ is also the second smallest element. Because $\min_2 S = n \ge 2q$, $\min_2 S/q = n/q \ge 2q/q  = 2 = |S|$. Let $m = \min S$. As we need $m/p \ge |S| = 2$, we must have $m\ge 2p$. Therefore $m$ can be any value from $2p$ to $n-1$. Hence we have $n-2p$ sets of $2$ elements.\\ \

\item For sets with at least three elements, we first find the range for the second smallest element. Let $\min_2S = i$ and $|S| = k$. Since there are $k-2$ elements bigger than $i$, $i\le n-k+2$. Because $\min_2 S/q\ge |S|$, we have $i\ge q k$. So, $q k\le i\le n-k+2$. Next, we find the upper bound for $k$. It follows from the fact that $q k\le n-k+2$, and thus we obtain $k\le \frac{n+2}{q+1}$. With $i$ and $k$ fixed, there are $i-p k$ choices for $\min S$ because $i = \min_2 S>\min S \ge p k$. Finally, we have $\binom{n-i-1}{k-3}$ choices to pick $k-3$ elements between $\min_2 S = i$ and $n$, so our formula is correct.
\end{itemize}
\end{proof}

\begin{proof}[Proof of Theorem \ref{keylem}]
For a nonempty, finite set $S$, define $S' := S\backslash \{\max S\}$. Clearly $M_{p, q, n+2q+1}$ is the union of three following disjoint sets:
\begin{itemize}
    \item [(a)] $A := \{S\in M_{p, q, n+2q+1}: n+2q\notin S\}$,
    \item [(b)] $B := \{S\in M_{p, q, n+2q+1}: n+2q\in S, S' - q \in M_{p,q, n+q}\}$, and
    \item [(c)] $C := \{S\in M_{p, q, n+2q+1}: n+2q\in S, S' - q \notin M_{p,q, n+q}\}$.
\end{itemize}


Let $\tau(S) = (S\backslash \{\max S\})\cup\{n+2q+1\}$. We compute $|A|$ by considering the map $\tau: M_{p, q, n+2q}\rightarrow A$.
The map is well-defined because
\begin{enumerate}
    \item for all $S\in M_{p, q, n+2q}$, $\tau(S)$ does not contain $n+2q$,
    \item $\tau$ does not change the cardinality of a set, while both the smallest and the second smallest of the set do not decrease.
\end{enumerate}
Clearly $\tau$ is one-to-one. We show that it is also onto. Let $U\in A$. If $|U| = 1$, that is $U = \{n+2q+1\}$, then $\tau(\{n+2q\}) = U$. If $|U| = 2$, we have $U = \{m,n+2q+1\}$ for some $2p\le m < n+2q$. Then $\tau(\{m,n+2q\}) = U$. If $|U|\ge 3$, then \begin{align*}\tau(\{n+2q\}\cup U\backslash \{n+2q+1\}) \ =\ U.\end{align*} Therefore $\tau$ is onto and thus, bijective. So, $|A| = |M_{p, q, n+2q}| = a_{n+q}$. \\ \

Let $\psi(S) = (S+q)\cup \{n+2q+1\}$. We compute $|B|$ by considering the map $\psi: M_{p,q, n+q} \rightarrow B$. Note that $\psi$ is well-defined because while $\psi$ makes the cardinality of a set increase by $1$, both the smallest and the second smallest increase by $q$. Clearly $\psi$ is one-to-one, and by the definition of $B$, it is also onto. Therefore $|B| = |M_{p, q, n+q}| = a_n$. \\ \

Finally, we compute $|C|$. Partition $C$ into $C_i$, where
$$C_i \ =\ \{S\in M_{p, q, n+2q+1}\, :\, n+2q\in S \mbox{ and }p|S| + i = \min S\},$$
for $0\le i\le q-p-1$. We show that $C = \cup_{i=0}^{q-p-1}C_i$. Let $F\in C_i$ for some $0\le i\le q-p-1$. We have
$$\min(F'-q)\ =\ \min F-q\ =\ p|F|+i-q\ <\ p |F| -p \ =\ p|F'-q|.$$
So, $F'-q \notin M_{p, q, n+q}$. Hence, $F\in C$. We have shown that $\cup_{i=0}^{q-p-1}C_i \subseteq C$. Now, let $E\in C$. Because $E\in M_{p, q, n+2q+1}$ and $E'-q\notin M_{p,q, n+q}$, it is straightforward to deduce that $\min (E'-q) < p|E'-q|$, which implies that $p|E|\le \min E < p |E| + (q-p)$. Therefore, $\min E = p |E| + i$ for some $0\le i\le q-p-1$. This shows that $C\subseteq \cup_{i=0}^{q-p-1}C_i$. We conclude that $C = \cup_{i=0}^{q-p-1}C_i$.

It remains to prove that $|C_i| = |M_{q,n}|$. Consider the map
\begin{align*}
    \phi_i: C_i &\longrightarrow M_{q, n}\\
            S   &\longrightarrow (S'\backslash \{\min S\}) - 2q.
\end{align*}
We show that $\phi_i$ is well-defined as follows. Let $F\in C_i$. Observe that
\begin{align*}q|\phi_i(F)| &\ =\ q|(F'\backslash \{\min F\}) - 2q|\ =\ q(|F| -2) = q|F|-2q\\
&\ \le\ {\rm min}_2 F - 2q \ =\ \min ((F'\backslash \{\min F\})-2q)\ =\ \min \phi_i(F).
\end{align*}
To see that $\phi_i$ is onto, let $G\in M_{q, n}$ and $H = \{p(|G|+2)+i\}\cup (G+2q)\cup \{n+2q+1\}$. We have $\min H = p(|G|+2) + i$ since
\begin{align*}p(|G| + 2)+i \ \le\ p(|G|+2)+(q-p)\ < \ p|G| + 2q \ \le\ \min (G+2q).\end{align*}
It follows that $H\in C_i$ because
\begin{align*}p|H| &\ =\ p(|G|+2) \ \le \ p(|G| + 2) + i \ =\  \min H,\ {\rm and}\\
q|H| &\ =\ q(|G|+2)\ \le\ \min G + 2q\ =\ {\rm min}_2 H.
\end{align*}
Clearly $\phi_i(H) = G$ and thus $\phi_i$ is onto. Since injectivity of $\phi_i$ is clear, $\phi_i$ is bijective. This shows that $|C_i| = |M_{q, n}|$ and so, $|C| = (q-p)|C_i| = (q-p)|M_{q,n}|$.

We conclude that
$$|M_{p, q, n+2q+1}| \ =\ |A| + |B| +|C| \ =\ |a_{n+q}|+|a_{n}|+(q-p)|M_{q,n}|.$$
\end{proof}



\newcommand{\etalchar}[1]{$^{#1}$}

\ \\

\noindent MSC2010: 11B39, 11Y55

\end{document}